\documentclass[12pt,a4paper,english]{article}

\usepackage[english]{babel}
\usepackage{amsmath,amsfonts,amssymb,amsthm}

\usepackage{authblk}
\title{Phase transition in random distance graphs on the torus}
\author[1, 2]{\normalsize Fioralba Ajazi}
\author[1 ]{\normalsize George M. Napolitano}
\author[1, 3]{\normalsize Tatyana Turova}

\affil[1]{\small Department of Mathematical Statistics, Faculty of Science, Lund University, S\"olvegatan 18, 22100, Lund, Sweden.}
\affil[2]{\small Department of Marketing and
	Operations Management, Faculty of Business and Economics, University of Lausanne, CH-1015, Switzerland.}
\affil[3]{\small Institute of Mathematical Problems of Biology, Russian Academy of Sciences, Institutskaja str., 142290, Pushchino, Moscow Region, Russia.}

\date{}                     
\newtheorem{theorem}{Theorem}[]
\newtheorem{prop}{Proposition}[]
\newtheorem{lemma}{Lemma}[]

\newtheorem{rem}{Remark}[]
\date{}

\newcommand\E{\operatorname{\mathbb E{}}}
\renewcommand\P{\operatorname{\mathbb P{}}}

\newcommand\Po{\operatorname{Po}}

\newcommand\Bin{\operatorname{Bin}}

\newcommand{\eqd}{\stackrel{d}{=}}

\begin{document}

\maketitle 

\begin{abstract}
We apply here methods  of inhomogeneous random graphs to
a class of random distance graphs. This provides an example  outside of the rank 1 models which is still solvable as long as the largest connected component is concerned. 
 In particular, we show that some random distance graphs  behave exactly as 
 the  classical Erd\H{o}s-R\'{e}nyi model not only in the supercritical phase (as was already known) but in the subcritical case as well.

\end{abstract}

\section{Introduction }

Random distance graphs, that is random graphs where  the vertices are in a metric space, and the connection probabilities  depend on the distance between the vertices, form a particular subclass of a general inhomogeneous random graph models \cite{irg}. The theory of inhomogeneous random graphs was founded in the
 seminal paper \cite{irg}, where the authors gave a common ground for  
numerous models introduced and studied previously almost independently. 

The graphs 
treated in \cite{irg} 
have the following common feature: 
 the edges are independent and the probability of edges, is, roughly speaking, of 
order $1/n$, where $n$ is the size of the graph. Shortly, each out of $n$ vertices is assigned a type, i.e., a value in some separable metric space ${\cal S}$. Given a set of such values $\{x_1, \ldots, x_{n}\}$ 
any
two vertices $i$ and $j$ are connected  with a probability
\begin{equation}\label{pIRG}
 p_{n}(i,j)=\min\left\{\frac{\kappa (x_i,x_j)}{n},1\right\},
\end{equation}
where $\kappa $ is a symmetric
nonnegative measurable function.

Random distance graphs assume usually that the vertices are in 
 $\mathbb{R}^d$  or $\mathbb{Z}^d$. Hence, their coordinates represent the types in the terminology of \cite{irg}.
The probability of the connections between any two vertices, say $u,v \in \mathbb{R}^d$,
in these models is a function of the distance between them, e.g.,  it decays with the distance (e.g., \cite{koz}, \cite{bp}, or percolation models \cite{penrose}, \cite{sp}), or simply it can be a step function (e.g., \cite{avin}). 
These models 
are often designed
to describe models of real world, as, e.g., various biological or social networks, where
the frequency, or the probability of connections between two vertices 
often depend on the distance. 
Therefore, the analysis of such models often focuses on properties such as diameter (\cite{avin}, \cite{koz}), cluster coefficient  (\cite{avin}), 
propagation of a bootstrap percolation processes  (\cite{koz}, \cite{bp}).

First models of spread-out percolation (see \cite{P} and \cite{penrose}, in particular) were studied even earlier than the theory in \cite{irg} was developed.

The relation between random graphs (on a finite set of vertices) and some spread-out percolation models (on a countable set of vertices) was exemplified in particular, by works \cite{penrose} and \cite{sp}. 
In \cite{sp}, the results of \cite{penrose} on the critical parameters for the existence of an infinite cluster in a class of spread-out percolation
where further generalized, and shown  to be a direct consequence of the general theory of inhomogeneous random graphs \cite{irg}.

Inspired by the models of inhomogeneous graphs \cite{irg}
the authors of \cite{rvdh} generalized  the model of \cite{AN} by introducing more inhomogeneity in the probabilities of connections (see details below). 

Most of the investigations of the random distance graphs are being done without much application of the theory \cite{irg} (not counting the Example 4.6 in 
\cite{irg} itself,
paper \cite{sp} is almost 
an exception). The reason, perhaps, 
 is that although the random distance graphs are 
examples of a general model \cite{irg}, they 
belong to the most complicated subclass 
of these models being outside of the so-called rank-1 case. Note that while the theory of \cite{irg} gives us the critical parameters for the emergence of the giant component and even describes the size of this component in the supercritical phase,  
there is no yet a general theory for 
the subcritical phase. 
The subcritical phase of the models which are not of rank 1
was studied only for some particular subclasses (see \cite{T1}), which do not include the present model. Furthermore, the critical phase was studied so far only 
for the rank-1 cases (\cite{A}, and later \cite{BHL}, \cite{BHL1}, and \cite{T11}).

Here we consider a model 
recently introduced in \cite{koz}. As we demonstrate below, it is 
 ``a simple homogeneous'' case (as Example 4.6 in \cite{irg}). 
We generalize model of  \cite{koz} in spirit of work \cite{rvdh}. This makes the model essentially inhomogeneous but still allows to get the exact 
asymptotics for the size of the largest connected component.

\section{The Model}
\label{sec:model}

Let  $N>1$ and let ${V}_N$ denote the set of vertices in the
2-dimensional discrete torus
$\mathbb{T}^2_N = (\mathbb{Z}/N\mathbb{Z})^2=\{1, \ldots, N\}^2$. 
Define 
the graph distance $d(u,v)$
between two vertices $u=(u_1,u_2)$ and $v= (v_1,v_2)$ in 
${V}_N$  
as follows 
\begin{equation}
	d(u,v) = d_N(|u_1-v_1|) + d_N(|u_2-v_2|),
\label{disT}
\end{equation}
where
\begin{equation*}
	d_N(i) = 
	\begin{cases}
		i, & 0 \leq i \leq N/2, \\
		N - i, & N/2 < i < N,
	\end{cases}
\end{equation*}
for $i \in \{0,\dots,N-1\}$.

Let $W$ be a nonnegative random variable, and let $W_v$, $v\in {V}_N$,
be $i.i.d.$ copies of $W$. 
Given the values $W_v$, $v\in {V}_N$, assume 
that between any two vertices $u,v \in {V}_N$ 
an edge is present independently of others and with probability
\begin{equation*}\label{p}
	p(u,v) = \min \left\{c \frac{W_u W_v}{N d(u,v)}, 1 \right\}, 
\end{equation*} 
where $c > 0$ is a parameter.

Denote $G_{N,W}$ the resulting graph on ${V}_N$. 

In the case of the constant $W\equiv 1$  it is exacly the graph
 introduced and studied recently in \cite{koz}. 
This model is inspired by the study of neuronal networks (see references in \cite{koz}). It also
has common features with other random graph models introduced previously in that 
area, see e.g., \cite{T}, where the same dynamics of activation was also approximated using the mean-field approach. On the other hand, this model  is closely related to the models considered in \cite{penrose}, \cite{sp} and \cite{rvdh}. 

The work \cite{koz} focuses on a bootstrap percolation process, which in this case models  the spread of the activation in a neuronal tissue. 
The authors of \cite{koz} also study the underlying graph and derive the order of its diameter, thus extanding the results in \cite{irg} for the graph with unbounded number of types.

Here we describe the phase transition in the  largest connected component.
The relations between our random graph 
and the percolation model of
\cite{rvdh} are exactly as shown  in \cite{sp} for the homogeneous case of random distance graphs.

\section{Results}

It was disclosed already in \cite{irg} (Example 4.6)  that in the  supercritical case 
a homogeneous distance graph has the same asymptotics for the size
of largest connected component as in the  
 classical Erd\H{o}s-R\'{e}nyi model. 
We prove that this result holds for 
 the subcritical case as well.

\begin{theorem}
\label{T1}
Let $G_N$ denote a random graph on ${V}_N$
with probability of connections 
\[p(u,v)=\min\left\{ \frac{c}{Nd(u,v)},1\right\}, \ \ \ \ \ \ u,v \in {V}_N,\]
and let $C$ denote the size of the largest connected component in $G_N$.
Set
\begin{equation*}\label{lacr}
\lambda=c \, 4\log 2.
\end{equation*}
Then  the following holds. 
\begin{itemize}
\item[i)]If  $\lambda<1$, we have that 
\begin{equation}\label{LLC conv}
 \frac{C}{\log (N^2)}\xrightarrow[]{P}\frac{1}{\lambda-1-\log \lambda}, \quad \text{as $N \rightarrow \infty  $}. 
\end{equation}
\item[ii)]  If $\lambda>1$ then
\begin{equation*}
	\frac{C}{N^2}\xrightarrow[]{P}\beta_\lambda, \quad \text{as $N \rightarrow \infty $},
\end{equation*}
where $\beta= \beta(\lambda)$ is the positive solution of $\beta=1-e^{\lambda \beta} $.
\end{itemize}
\end{theorem}

As we noted above, only the case $ii)$ of this theorem follows by the results of \cite{irg}.  

\begin{rem}One may observe a certain redundancy here, as the statements 
$ii)$ and $i)$ of
Theorem \ref{T1} are particular cases of the following below Theorems \ref{T2}
and \ref{T3}, respectively. However, stated separately
Theorem \ref{T1} makes it clear that  asymptotically the largest connected component in $G_N$ behaves as the one in the 
Erd\H{o}s-R\'{e}nyi graph $G_{n,p}$ with $n=N^2$ and $p=\lambda/n$. Furthermore, it is plausible to conjecture (but we do not prove it here) that even in the critical case the model has the same asymptotics for the largest component as the one in $G_{n,p}$ with $p=1/n$.
\end{rem}

The following theorem treats the general case (\ref{p}).

\begin{theorem}
\label{T2}
Assume, that 
\begin{equation}\label{cW}
\mathbb{E}W^2= \int_0^{\infty}   x^2  \mu_W(dx)<\infty.
\end{equation} 
Let $C(G_{N, W})$ denote the size of its  largest connected component in $G_{N, W}$, and denote again 
\[\lambda=c \, 4\log 2.\]
Then 
\begin{equation}\label{O41}
 \frac{C(G_{N, W})}{N^2}\xrightarrow[]{P} \int_0^{\infty}  \beta(x)  \mu_W(dx)=:{\hat \beta},
\end{equation}
where $\beta(x)$ is the maximal solution to 
\begin{equation}\label{O38}
 f(x)=1-e^{x\lambda \int_0^{\infty}  y f(y)  \mu_W(dy)}.
\end{equation}
Furthermore, ${\hat \beta}>0 $ if and only if
\begin{equation}\label{lacr1}
 \lambda \mathbb{E}W^2 >1.
\end{equation}
\end{theorem}

Note that  the critical parameter $\lambda \mathbb{E}W^2 $ in  Theorem \ref{T2} 
is similar (in fact it has exactly the same meaning of a certain average of the degree of a vertex) to the lower bound derived in Theorem 4.1 in \cite{rvdh} to 
provide the necessary conditions for the percolation. 

Theorem \ref{T2} follows essentially by the general theory of \cite{irg} as we explain below. 
It tells us  that the asymptotics of the 
largest component in 
$G_{N, W}$ behaves as  the one in the rank-1 random graph on $V_N$ defined by the following probabilities of connections between any $u,v\in V_N$
\begin{equation}\label{pr1}
	p_1(u,v) = \min\left\{\lambda \frac{W_u W_v}{N^2}, 1 \right\}.
\end{equation}

In this case the largest connected component in the subcritical case is sensitive to the tail of the distribution of  $W$. It is known that in the rank-1 case the size of the largest component varies  between  polynomial (as in \cite{J2}) and logarithmic  (as in \cite{T2}) orders. This behaviour is projected to the model under consideration here. 
We shall state a particular case of logarithmic order to show  the similarities with Theorem \ref{T1}.

\begin{theorem}
\label{T3}
 Assume  that for some positive $\varepsilon$ 
\begin{equation}\label{cW2}
 \mathbb{E}e^{\varepsilon W}<\infty.
\end{equation}
If also 
\begin{equation}\label{as4}
\lambda \mathbb{E}W^2<1,
\end{equation}
 there exists a unique $y>1$ 
which satisfies
\begin{equation}\label{O36}
 y=\frac{1}{\lambda M }
\frac{\mathbb{E}\left( W e^{\lambda M(y-1)W }\right)}{
\mathbb{E}\left(W^2 e^{\lambda M(y-1)W }\right)},
\end{equation}
where $M=\mathbb{E}W$. Let $C$ be the size of the largest connected component of $G_{N,W}$. Then one has
\[
 \frac{C}{\log (N^2)}\xrightarrow[]{P}\frac{1}{\log \gamma}, \quad \text{as $N \rightarrow \infty  $},
\]
where
\[\gamma:= \frac{1}{\lambda
\mathbb{E}\left(W^2 e^{\lambda M(y-1)W }\right)}>1.\]
\end{theorem}

\section{Proofs}

\subsection{Random distance graph via inhomogeneous random graphs}

Rescale the set ${V}_N$ as follows
 $${V}_N \rightarrow {\widetilde V}_N  = \{(u_1/N,u_2/N): (u_1,u_2)\in {V}_N\}.$$
Hence,  ${\widetilde V}_N$ is a subset of $N^2$ vertices in a
 continuous torus $\mathbb{T}^2:=(\mathbb{R}/\mathbb{Z})^2$. 
Let $\mu_{\cal L}$ denote the Lebesgue measure on $\mathbb{T}^2$, and let 
$\mu_{W}$ be the Borel measure on $\mathbb{R}_+$ induced by the random variable $W$. 
Denote $\mathcal{S}:=\mathbb{T}^2 \times \mathbb{R}_+$, and define
a product measure $\mu = \mu_{\cal L} \times \mu_{W} $ on this space. 
Then the triple $\mathcal{V}:=(\mathcal{S}, \mu, \{(v,W_v): v\in {\widetilde V}_N\})$ satisfies the definition of a generalized vertex space of \cite{irg}, i.e., for any
Borel  set $A\subseteq {\cal S}$ the following convergence holds:
\begin{equation*}
\frac{|\{v: (v,W_v) \in A \}|}{N^2}\xrightarrow{P}\mu(A).
\end{equation*}

Define now  for $u\neq v, \ u,v \in \mathbb{T}^2 $ the kernel function
\begin{equation}\label{ka}
\kappa_1(u,v):= \frac{1}{\rho(u,v)}, 
\end{equation}
where for any $u=(u_1, u_2), v=(v_1, v_2) \in  \mathbb{T}^2$
\[\rho(u,v)  =\rho_1(|u_1-v_1|)+\rho_1(|u_2-v_2|) \]
with
\begin{equation*}
\rho_1(a)  =
	\begin{cases}
a, & 0\leq a\leq \frac{1}{2}, \\ 
1-a, & \frac{1}{2}< a\leq 1.
	\end{cases}
\end{equation*}
Denote also here $\kappa_2(x,y):= xy$ a product function on $\mathbb{R}^2_+$.

Finally, using the kernel 
\begin{equation}\label{ka1}
\kappa((u,x),(v,y)):= \kappa_1(u,v)\kappa_2(x,y),
\end{equation}
$(u,v), (x,y) \in \mathcal{S},$
we define a random graph $G^\mathcal{V}\left(N^2, \kappa\right)$ 
(we follow the notation of \cite{irg}) 
on the set of vertices $ {\widetilde V}_N $ with given types 
\[{\widetilde {\bf V}}_N:=\{(v,w_v), v \in {\widetilde V}_N \} = \{(v_i,w_i)\}_{i=1}^{N^2}\subset \mathcal{S}\]
by setting an independent  edge between any pair of vertices 
${\bf x}_i, {\bf x}_j \in {\widetilde {\bf V}}_N$ 
 with probability (recall eq. (\ref{pIRG}))
\begin{equation}\label{pIRG1}
\widetilde{p} \left({\bf x}_i, {\bf x}_j \right) := 
\min 
\left\{ 
c \frac{\kappa\left( {\bf x}_i, {\bf x}_j\right)}{N^2}  ,1 
\right\}.
\end{equation}

\begin{prop}\label{p1}
The model $G_N$ is equivalent to the inhomogeneous
random graph model $G^{\mathcal{V}}(N^2, \kappa)$.
\end{prop}
\begin{proof}
Setting 
$\widetilde{v}=v/N$ and  $w_v=w_{\widetilde{v}}$
 for any $ v \in {V}_N$,   the probability of connection  (\ref{p})  is equivalent to 
\begin{equation}\label{eq1}
	p(u,v) = \min 
\left\{c \frac{w_u w_v}{N d(u,v)}, 1 
\right\}=\min 
\left\{c \frac{\kappa_1\left(\widetilde{u}, \widetilde{v} \right) w_{\widetilde{u}} w_{\widetilde{v}}}{N^2} , 1 
\right\}
\end{equation}
\[
=\widetilde{p}\left( (\widetilde{v}, w_{\widetilde{v}}),
(\widetilde{u}, w_{\widetilde{u}}) \right).
\]
Hence, 
there is a connection between any two vertices $ u,v\in  {V}_N$  of $G_N$,
if and only if there is a connection between the corresponding vertices 
$\widetilde{u}=u/N, \widetilde{v}=v/N$ of $G^{\mathcal{V}}(N^2, \kappa)$. 
\end{proof}

It is straightforward to check  that the  kernel $\kappa$ is \emph{graphical}, 
since it is continuous,
$\kappa \in L^1(\mathcal{S}\times \mathcal{S},\mu\times \mu)$, and the
number of edges in the graph
$e(G^\mathcal{V}(N^2,\kappa))$ satisfies the following convergence
\begin{equation}\label{O37}
\frac{1}{N^2}\E e(G^\mathcal{V}(N^2,\kappa))\rightarrow \frac{1}{2}\int_{\mathcal{S}}\int_{\mathcal{S}} \kappa(u,v)dudv.
\end{equation}

\subsection{Proof of Theorem \ref{T2}}

Proposition \ref{p1} together with  (\ref{O37})  enables  us to apply some of the results of \cite{irg}  to our case.
In particular, we can approximate the size of the connected component by the total progeny of a multi-type Galton-Watson branching process $\cal B({\bf x})$, with type space $\cal S$, where 
the single ancestor has type ${\bf x}$, and 
the number of offspring of type ${\bf y}$ of each  individual of type ${\bf x}\in {\cal S}$
has Poisson distribution with intensity $\kappa({\bf x},{\bf y})\mu(d{\bf y})$. Denote here
$\beta_{\kappa}({\bf x})$ and ${\cal X}({\bf x})$, correspondingly,
the survival probability and the 
size of the total progeny of this branching process with the ancestor of type ${\bf x}$. 

Following $\cite{irg}$,  let us define the integral operator $T_{\kappa}$  
\[
(T_{\kappa} f)({\bf x}):= \int_\mathcal{S}\kappa({\bf x},{\bf y})f({\bf y})d\mu({\bf y}),
\]
for all measurable functions $f$ (when the integral is defined) on $\cal S$, and define the norm of 
$T_{\kappa}$ by 
\begin{equation}\label{deT}
\parallel T_{\kappa}\parallel :=\sup\{\parallel T_{\kappa}f\parallel_2 :f\geq 0, \parallel f \parallel_2\leq 1 \}.
\end{equation}
Then, by  Theorem 3.1 in  \cite{irg} (whose applicability here is granted by  
Proposition \ref{p1}) we immediately get
\begin{equation}\label{O40}
 \frac{C(G_{N, W})}{N^2}\xrightarrow[]{P} \int_{\cal S}  \beta_{\kappa}({\bf x})
 \mu(d{\bf x} ).
\end{equation}
Moreover,  it is also proved in \cite{irg} that the survival probability $\beta_{\kappa}$ is the maximal solution to 
\begin{equation}\label{O34}
 f({\bf x})=1-e^{-(T_{\kappa}f)({\bf x})}, \quad {\bf x}\in {\cal S},
\end{equation}
and that ${\hat \beta}_{\kappa}>0$ if and only if
\begin{equation}\label{conT}
\parallel T_{\kappa}\parallel >1.
\end{equation}

Observe that 
it follows directly from the  symmetry of our model that 
the survival probability $\beta_{\kappa}({\bf x})$ where ${\bf x}=(u,x)\in \mathbb{T}^2 \times \mathbb{R}_+$,
does not depend on $u\in \mathbb{T}^2$, but it is only a function of $x\in \mathbb{R}_+$.  Hence we shall simply write the survival probability as $\beta_{\kappa}({\bf x})=\beta_{\kappa}({ x})$, which by (\ref{O34}) is the maximal solution to 
\begin{equation}\label{surv}
 f({x})=1-e^{-\lambda \int_{0}^{\infty}xyf(y)\mu_W(dy)}, \ \ {x}\in {\mathbb{R}}_+,
\end{equation}
which is exactly equation  (\ref{O38}). This together with (\ref{O40}) yields
(\ref{O41}).

We are left to prove  (\ref{lacr1}). Firstly, one could use
straightforwardly the definition  (\ref{deT}) to derive
\begin{equation}\label{O39}
\parallel T_{\kappa}\parallel = \lambda \mathbb{E}W^2,
\end{equation}
which together with (\ref{conT}) would give us (\ref{lacr1}).
However, 
it is easier to derive  (\ref{lacr1})
directly relating the defined above multi-type branching process to a 
a certain homogeneous Galton-Watson process as it was explained, e.g., in \cite{T2}.

Let us introduce yet another branching process $\mathcal{B}_1(x)$ 
with type space $\mathbb{R}_+$, where the single ancestor has type $x$, 
and number of 
offspring of type $y\in \mathbb{R}_+$ of any individual of type $x \in \mathbb{R}_+$ has Poisson distribution with intensity $\lambda xy\mu_W(dy)$. By using the same analysis as for $\cal B$ we can define the survival probability $\beta_{\kappa}^1(x)$ of ${\cal B}_1(x)$ as the maximum solution of (\ref{surv}). Therefore  in notations of Theorem \ref{T2} it holds that 
$\beta({x})=\beta_{\kappa}^1(x)$, and
\begin{equation}\label{Nov1}
\beta_{\kappa}({\bf x})=\beta_{\kappa}^1(x)=\beta({x}), 
\end{equation}
 for any ${\bf x}=(u,x)$, and for all $u\in \mathbb{T}^2$.

Finally, we define one more auxiliary branching process  ${\cal B}_2$, this time a 
homogeneous Galton-Watson process. This process starts with one single ancestor, and its
offspring distribution $\tilde{Y}$ has a compound Poisson distribution 
\begin{equation*}
\Po( \widetilde{W} \lambda \E(W)),
\end{equation*}
where random variable 
 $\widetilde{W}$ has the following so-called "size-biased'', distribution:
\begin{equation*}
\mu_{\widetilde{W}}(dy):=\frac{y\mu_W( dy)}{\E W }.
\end{equation*}

  Let us denote   ${\cal X}_1(x)$ and ${\cal X}_2$ the total progeny of ${\cal B}_1(x)$ and ${\cal B}_2$,  respectively. It is proved in  Section $2.2$ in \cite{T2},  that ${\cal X}_1(\widetilde{W}) $, and ${\cal X}_2$ are equal in distribution, i.e.,
\begin{equation}\label{Nov2}
{{\cal X}}_1({\widetilde{W}})\stackrel{d}{=}{\cal X}_2.
\end{equation}
 Hence,  the critical parameters 
for the corresponding survival probabilities
for these two branching processes ${\cal B}_1(\widetilde{W})$ and ${\cal B}_2$ are the same. In the case of a homogeneous process ${\cal B}_2$ the necessary and sufficient condition for a positive 
survival probability 
is simply 
$\E(\tilde{Y})=\lambda \E(W^2)>1$. Therefore (\ref{Nov2}) yields that
$\P(\widetilde{{\cal X}}_1({\widetilde{W}})=\infty)>0$ if and only if 
$\lambda \E(W^2)>1$.

It follows by the properties of Poisson distribution that
 the type of a randomly chosen offspring of the ancestor 
in the process ${\cal B}_1(x)$
 has distribution 
$\widetilde{W}$. Hence, for any $x$ the process ${\cal B}_1(x)$ survives with a positive probability, i.e., $\beta_{\kappa}^1(x)>0$,   if ${\cal B}_1(\widetilde{W})$ survives with a positive 
probability. Since $\beta_{\kappa}^1$ is the maximal solution to (\ref{surv}), i.e., to 
 (\ref{O38}), it follows that $\hat{\beta}>0$  (see (\ref{O41})) in this case.

On the other hand, if ${\cal B}_1(\widetilde{W})$ survives with probability zero, then the
equality
\[ 0=\P(\widetilde{{\cal X}}_1({\widetilde{W}})=\infty)
= \int_{\mathbb{R}_+}\P(\widetilde{{\cal X}}_1({x})=\infty)\mu_{\widetilde{W}}(dx)
= \int_{\mathbb{R}_+}\beta_{\kappa}^1(x)\mu_{\widetilde{W}}(dx)\]
implies that $\beta_{\kappa}^1=0$ $\mu_{\widetilde{W}}$ -$a.s.$, and hence, $\mu_{{W}}$ $-a.s.$
Since $\beta_{\kappa}^1$ is the maximal solution to (\ref{surv}), i.e., to 
 (\ref{O38}), it follows that in this case $\hat{\beta}=0$. 

Summarizing we get that $\hat{\beta}>0$ if and only if $\lambda \E(W^2)>1$. In turn this yields $\|T_{\kappa}\|=\lambda \E(W^2)$. This proves the theorem.

\subsection{Proof of Theorem \ref{T1}}

\subsubsection{Exploration algorithm.}

Let $v$ be an arbitrarily fixed vertex of a graph $G$,  and let $C_v$ denote the connected component containing $v$. 
We use a standard procedure to reveal $C_v$, called exploration algorithm. This is defined as follows.

In the course of exploration the vertices of $G$ can be in one of
the  three states: \textit{active}, \textit{saturated} or \textit{neutral}. 
At time $i = 0$ the vertex $v$ is set to be active, while
  all the other vertices are set to be neutral. 
We denote by $S_i$ the set of active vertices at time $i$. Hence, $|S_0| = 1$. 
The state of a vertex changes during the exploration of $C_v$ as follows. 

At each time step $i \geq 1$, we choose an active vertex in
$S_{i-1}$ uniformly at random and denote it by $v_i$. Then, consider each neutral vertex $u$: if
$u$ is connected to $v_i$ then $u$ becomes active, otherwise $u$ stays neutral. 
After searching the entire set of neutral vertices the vertex $v_i$ becomes 
saturated. This finishes the $i $-th step of the exploration algorithm.

The process stops when there are no more active vertices, i.e. at the first 
time $i$ when $|S_{i}|=0$, that is at time
\begin{equation}
\label{defT}
T=\min\{i\geq 1: |S_{i}|=0 \}.
\end{equation}
At this time all considered vertices are saturated and they do not have any connection to the neutral vertices. Hence,  $C_v$ coincides with the set of saturated vertices, that is $|C_v|=T$.

Let $X_i$ denote the number of vertices becoming active at the $i$-th step. Then the following recursion holds 
\begin{equation}
\label{eq:rec_active}
\begin{split}
|S_0|&=1,\\
|S_i|&=|S_{i-1}|+X_i-1=X_1+\dots+X_{i}-(i-1).
\end{split}
\end{equation}
Correspondingly, we can rewrite  $T$, defined in eq. (\ref{defT}), as follows
\begin{equation}\label{defT1}
T=\min\{i \geq 1: X_1+\dots + X_i=i-1\}.
\end{equation}

\subsubsection{Subcritical case}

\label{sec:sub}

In this section we assume that  $\lambda < 1$ and we prove part $i)$ of Theorem \ref{T1}.

\paragraph{Lower bound.}
We start by finding a lower bound on $C$, the size of the largest connected component. Namely, we prove that for any positive $\varepsilon$
\begin{equation}\label{LLC low}
{\mathbb{P}}\left\{\frac{C}{\log N} \geq \frac{2}{1-\lambda-\log \lambda}+\varepsilon\right\}\longrightarrow 0, \text{ as $N\longrightarrow \infty$}.
\end{equation}

\begin{proof}
The proof is based on the exploration algorithm described above. We also use essentially the geometry of the discrete torus with the distance defined in eq. (\ref{disT}). 
Recall, in particular,  that the number $N_r$ of vertices at distance $r$ from any  given vertex, 
for $N$ odd is given by 
\begin{equation}\label{ndr}
N_r=\begin{cases}
4r, &  1\leq r\leq \lfloor N/2 \rfloor,\\
4(N-r), & \lfloor N/2 \rfloor <r\leq N;
\end{cases}
\end{equation}
 while for $N$ even, it is given by 
 \begin{equation}\label{ndr2}
N_r= \begin{cases}
4r, & 1 \leq r < N/2,\\
2(N-1), &  r=\frac{N}{2},\\
4(N-r), & N/2 < r < N,\\
1, & r=N.
 \end{cases}
 \end{equation}

Recall that the vertices becoming active at the \textit{i}-th  step are connected to the vertex $v_i$.
Let $X_{i,r}$ denote the number of vertices at distance $r$ from vertex $v_i$, which
become active at the $i$-th step.
Hence, 
\begin{equation}
	X_i = \sum_{r = 1}^N X_{i,r}.
\label{Xr}
\end{equation}
Let $U_i$ denote the number of active and saturated vertices at time $i$
(in other words, $U_i$ is  the number of vertices revealed by time $i$).
In particular, by  eq. (\ref{eq:rec_active}) we have
\begin{equation}
	U_i = |S_i| + i.
\label{eq:used_v}
\end{equation}
Correspondingly, for any vertex $u$ let $U_{i,r}(u)$ be the number  
of active and saturated vertices at time $i$, which are
 at distance $r$ from 
$u$. In particular, for any $i\geq 1$ and any vertex $u$ it holds that
\begin{equation*}
	\sum_{r=1}^{N}U_{i,r}(u)=U_{i}.
\end{equation*}

The number $X_{i,r}$ depends on the number $U_{i-1,r}(v_i)$ of active and saturated vertices at time $i-1$ which are at distance $r$ from $v_i$, in the following way 
\begin{equation}\label{O20}
X_{i,r}|_{U_{i-1,r}(v_i)} \in \Bin(N_r-U_{i-1,r}(v_i),p_r),
\end{equation}
where we use the notation 
\begin{equation*}
	p_r=\frac{c}{Nr} = p(u,v), \quad \text{ if $d(u,v)=r$}.
\end{equation*}

\begin{rem}\label{Rem1} In eq. (\ref{O20}) and elsewhere we write a random parameter for a distribution with the usual meaning that the distribution is defined conditionally on a given value of the parameter. 
\end{rem}

Let us introduce  the random variables  
\begin{equation*}
Z_{i,r} \in \Bin(U_{i-1,r}(v_i),p_r),
\end{equation*}
and 
\begin{equation*}
X_{i,r}^+ = X_{i,r} + Z_{i,r} \in \Bin (N_r,p_r).
\end{equation*}
Then, we define 
\begin{equation*}
X_i^+ := \sum_{r=1}^{N} X_{i,r}^+.
\end{equation*}

If a random variable $\xi$ stochastically dominates $\eta$ we denote this by $\eta \preceq \xi$.

It is clear from the above definition that $X_{i,r} \preceq X_{i,r}^+$, and, correspondingly,  $X_i \preceq X_i^+$. Therefore,  $$|S_i| \preceq S_i^+:= X_1^+ + \dots+ X_{i}^+ - (i-1).$$

Notice that the largest connected component has size larger than $k$ if and only if there is a component of size at least $k$. Let $|C_v| $ denote 
the size of the component $C_v$.
Then
\[
\P \{C\geq k\} = \P \{ \exists v : |C_v| \geq k \} = \P \{ \bigcup_{v \in V} \{|C_v| \geq k \}.
\]
It follows simply by the symmetry of the model that the random variables $|C_v|$
have equal distribution for all vertices $v$. This allows us to derive from the last formula the following bound
\begin{equation}\label{LLC}
\P \{C\geq k\}
\leq N^2 {\P}\{|C_v|\geq k \},
\end{equation}
for any arbitrarily fixed vertex $v$.

By the exploration algorithm we have 
that the probability for a component $C_v$ to be larger or equal to $k$ is 
equal to the probability of having active vertices in each of the $k-1$ steps of the exploration, Hence, 
\begin{equation}
\label{eq:Cv_bound}
\begin{split}
{\P}\{|C_v|\geq k \}&= \P\{|S_t| > 0, \forall t\leq k-1\}\\
&\leq \P\{S_t^+ > 0,\forall t\leq k-1\}\\
&\leq \P\{S_{k-1}^+ > 0  \}=\P\{\sum_{t=1}^{k-1} X_t^+ - (k-2)>0  \}.
\end{split}
\end{equation}

We use the coupling method described in \cite{lindvall} for finding stochastic bounds on $X_i^+$. We have that $X_{i,r}^+$ is stochastically bounded from above by a random variable $Y_{i,r} \eqd \Po(-N_r \log (1-p_{r})) $, i.e. $X_{i,r}^+ \preceq Y_{i,r}$. Therefore,  we can stochastically bound $X_i^+$ by a Poisson random variable, as follows
\begin{equation}\label{ln2}
	\begin{split}
		X_i^+ & \preceq \sum_{r=1}^{N} Y_{i,r} \in \Po\left(\sum_{r=1}^{N}-N_r \log(1-p_r)\right) = \Po(\lambda_N),
	\end{split}
\end{equation}
where 
\begin{equation}\label{ln}
 \begin{split}
\lambda_N&=\sum_{r=1}^{N}-N_r \log(1-p_r)=\sum_{r=1}^{N}N_r (p_r+o(p_{r}))\\
&=\sum_{r=1}^{\lfloor N/2 \rfloor} 4r(p_{r}+o(p_{r}))+\sum_{r=\lfloor N/2 \rfloor +1}^{N}4(N-r)(p_{r}+o(p_{r}))\\
&=\lambda -\frac{2c}{N} +o\left(\frac{1}{N}\right).
\end{split}
\end{equation}

Let  $Y_i$, $i\geq 1$, be the $i.i.d.$ random variables with
 $\Po(\lambda_N)$ distribution. Then we derive, using (\ref{eq:Cv_bound})
and (\ref{ln2}) with  (\ref{ln}), the following upper bound for the probability in  (\ref{LLC}):
\begin{equation}\label{c1}
\begin{split}
{\mathbb{P}}\{C \geq k\} & \leq N^2\P\left\{\sum_{t=1}^{k-1} X_t^+ >k-2 \right\}
\leq N^2 \P \left\{\sum_{t=1}^{k-1}Y_t>k-2 \right\}.
\end{split}
\end{equation}

Using  Chebyshev's inequality in (\ref{c1}), for any $h>0$,  we have that
\begin{equation}\label{c1_end}
\begin{split}
{\mathbb{P}}\{C \geq k\}& \leq N^2 \P \left\{\sum_{t=1}^{k-1}Y_t>k-2 \right\}\\
& \leq \frac{N^2 \prod_{t=1}^{k-1}{\mathbb{E}}e^{h Y_t}}{e^{h(k-2)}}\\
&=N^2 e^{-h(k-2)}\prod_{t=1}^{k-1}e^{\lambda_N(e^h-1)}\\
&=N^2 e^{-h(k-2)}e^{(k-1)\lambda_N(e^h-1)}.
\end{split}
\end{equation}
The last formula attains its minimum at $h=\log\left(\frac{k-1}{k\lambda }\right)$, where it equals to 
\[
N^2e^{k(1-\lambda+\log \lambda)+k o(1)}.
\]
Therefore setting $k=\left(\frac{2}{\lambda-1-\log \lambda}+\varepsilon \right)\log N$ in (\ref{c1_end}),  we find that
(\ref{LLC low}) holds.
\end{proof}

\paragraph{The upper bound.}

To complete the proof of (\ref{LLC conv}) we will prove that for any $\varepsilon > 0$

\begin{equation}\label{LLC upp}
{\mathbb{P}}\left\{\frac{C}{\log N} \leq \frac{2}{1-\lambda-\log \lambda}-\varepsilon\right\}\longrightarrow 0 \,\,\,\,\text{as $N\longrightarrow \infty$}.
\end{equation}

\begin{proof}

Before proceeding to the proof of (\ref{LLC upp}) we 
derive a useful result, which roughly speaking tells us 
that removing an arbitrary set of 
$o(N^2)$ vertices from ${V}_N$ 
does not change some, asymptotically as $N \to \infty$, the expected degree of a vertex.

\begin{lemma}
\label{lem:subtract}
Let $n_r$, $r=1,\dots, N$, with $0 \leq n_r \leq N_r$, be an arbitrary sequence 
such that  $$\sum_{r=1}^{N}n_r=o(N^2).$$ Then  
\begin{equation}\label{on2}
\frac{1}{N}\sum_{r=1}^{N}\frac{n_r}{r}\rightarrow 0\quad \text{as $N\rightarrow \infty$}.
\end{equation}
\end{lemma}
\begin{proof}
We prove the lemma by contradiction. Assume there exists a constant $c > 0$ such that for any $M \in \mathbb{N}$ there exists $N \geq M$ such that 
\begin{equation}
\frac{1}{N}\sum_{r=1}^{N}\frac{n_r}{r}\geq c.
\label{eq:lim_neg}
\end{equation}

Let $ 0 < \delta < \min\{4,c\}$ and define the sets $\mathcal{N}_\delta$ and its complementary $\overline{\mathcal{N}}_\delta$ as follows
\begin{equation*}
\begin{split}
&\mathcal{N}_\delta=\{r \in \{ 1,\dots, N \} : n_r \geq \delta r \} \\
&\overline{\mathcal{N}}_\delta=\{r \in \{ 1,\dots, N \} : n_r < \delta r   \}.
\end{split}
\end{equation*}
Noting that from (\ref{ndr})-(\ref{ndr2}) we have $n_r \leq N_r \leq 4 r$, for any $0 \leq r \leq N$, from (\ref{eq:lim_neg}) it follows that
\begin{equation*}
\begin{split}
c & \leq  \frac{1}{N}\sum_{r=1}^{N} \frac{n_r}{r} = \frac{1}{N} \left( \sum_{r \in \mathcal{N}_\delta} \frac{n_r}{r} + \sum_{r\in \overline{\mathcal{N}}_\delta } \frac{n_r}{r} \right)\\
& \leq \frac{1}{N} \left( \sum_{r\in \mathcal{N}_\delta} 4 + \sum_{r\in \overline{\mathcal{N}}_\delta } \delta  \right)\\
& = \frac{1}{N} \left( 4 |N_\delta| + \delta \overline{\mathcal{N}}_\delta \right) \\
& = \delta + \frac{4-\delta}{N} |N_\delta|.
\end{split}
\end{equation*}
In particular, we have
\begin{equation*}
	|N_\delta| \geq \frac{c-\delta}{4-\delta} N,
\end{equation*}
and therefore
\begin{equation*}
\sum_{r = 1}^{N} n_r \geq \sum_{r\in \mathcal{N}_\delta} n_r \geq \sum_{r\in \mathcal{N}_\delta} \delta r \geq \delta \sum_{r=1}^{|\mathcal{N}_\delta|} r  \geq \frac{\delta}{2} |\mathcal{N}_\delta|^2
\geq \frac{\delta}{2} \left( \frac{c-\delta}{4-\delta} \right)^2 N^2,
\end{equation*}
which contradicts the assumptions.
\end{proof}

Now we can prove (\ref{LLC upp}). 
We shall follow the construction used already in \cite{T2}.
For any vertex $v$ let $V(C_v)$ 
denote here the set of vertices of the component $C_v$.
Observe that $G_N$ can be decomposed into pairwise disjoint connected components
as follows. Set $\tilde{v}_1 = v $. Then, given $C_{\tilde{v}_1}, \ldots, 
C_{\tilde{v}_k},
$ for $k\geq 1$
choose a vertex $\tilde{v}_{k+1}$ 
uniformly in ${V}_N \setminus \cup_{i=1}^k
V(C_{\tilde{v}_i})$, unless the last set is empty, in which case we stop the algorithm.
 The graph $G_N$ is thus 
decomposed into 
pairwise disjoint components
 $C_{\tilde{v}_1}, \dots , C_{\tilde{v}_M}$, with $M=M(N)$. Note that here $M$ is a bounded random variable, $1\leq M \leq N^2$, denoting the number of disjoint connect components  in $G_N$. 

Fix $\varepsilon>0$ arbitrarily  and  denote  
$K=\left(\frac{2}{\lambda-1-\log \lambda}+\varepsilon \right)\log N$. Then we define the event
\begin{equation*}
 E_N=\{C\leq K\}.
\end{equation*}
Recall that it 
 follows
from (\ref{LLC low}) that $$\P \{\bar{E}_N\}\rightarrow 0$$ as $N\rightarrow \infty$. This yields
\begin{equation}\label{c_m}
\begin{split}
\P \{C\leq k \} & =  \P \{ |C_{\tilde{v}_1}| \leq k,\dots, | C_{\tilde{v}_M}| \leq k \}\\
& \leq \P \{ |C_{\tilde{v}_1}| \leq k,\dots, |C_{\tilde{v}_M}| \leq k | E_N \} + o(1).
\end{split}
\end{equation}
Note that since conditionally on $E_N$ the largest connected component is smaller than $K$, it follows that $MK\geq N^2$. Hence, 
 for any $$M_1\leq \frac{N^2}{K}\leq  M$$ the following bound holds from the probability in 
  (\ref{c_m}):
\begin{equation}
\label{eq:prod}
\begin{split}
&\P \{ |C_{\tilde{v}_1}| \leq k,\dots, |C_{\tilde{v}_M}| \leq k| E_N\}
\\
&\leq  \prod_{i=1}^{M_1} \left. \P \{ |C_{\tilde{v}_i}| \leq k \right| |C_{\tilde{v}_1}| \leq k,\dots, |C_{\tilde{v}_{i-1}}| \leq k,E_N \}.
\end{split}
\end{equation}

Let $\mathcal{V}_0$ be an arbitrary  set of  $M_1 K$ nodes, 
$u$ be a vertex in ${V}_N \setminus \mathcal{V}_0$,
 and let $\tilde{C}_u = \tilde{C}_u(\mathcal{V}_0)$ denote 
the connected component containing $u$ 
constructed 
precisely as the original $C_v$
 but on the vertex set $ {V}_N \setminus \mathcal{V}_0$. 

Then, each factor in eq. (\ref{eq:prod}) can be uniformly bounded as follows
\begin{equation*}
\P \{ |C_{\tilde{v}_i}| \leq k| |C_{\tilde{v}_1}| \leq k,\dots, |C_{\tilde{v}_{i-1}}| \leq k,E_N \} \leq \max_{ \substack{\mathcal{V}_0\subseteq {V}_N:|\mathcal{V}_0|= M_1K\\ u\in {V}_N\setminus \mathcal{V}_0}} \P \{ |\tilde{C}_u| \leq k \},
\end{equation*}   
where we simply used the fact that on a smaller set of vertices the components are smaller as well. 
Therefore, from (\ref{eq:prod}) it follows that
\begin{equation}\label{c1l}
\begin{split}
\P \{C\leq k \}&  \leq \left(\max_{ \substack{\mathcal{V}_0:|\mathcal{V}_0|=M_1K\\ u\in {V}_N\setminus \mathcal{V}_0}} \P \{ |\tilde{C}_u| \leq k \} \right)^{M_1}.\\
\end{split}
\end{equation}

In the following, fix set $\mathcal{V}_0$ arbitrarily but so that 
 $|\mathcal{V}_0| = M_1K=o(N^2)$. Fix a  vertex $u \not\in \mathcal{V}_0$ arbitrarily, and construct the
 component $\tilde{C}_u$  on the vertex set $V_N \setminus \mathcal{V}_0$ 
as it is described in the exploration algorithm. Let us denote here $u_1, u_2, \ldots, $ the sequence of saturated vertices. 

Define $n_r^0(u)$  to be the number of nodes in $\mathcal{V}_0$
which are at distance $r$ from $u$, so that $0 \leq n_r^0(u) \leq N_r$ and $\sum_{r=1}^{N} n_r^0(u) = |\mathcal{V}_0|$ for any $u$.

In analogy with the notion used previously, let here $\tilde{U}_{i}$ denote the number of active and saturated vertices at step $i$ in the exploration process (cfr. eq. (\ref{eq:used_v})),  and $\tilde{U}_{i,r}(w)$ be the number 
of those vertices
at distance $r$ from the  vertex $w$. Let then $n_{i,r}^0=n_r^0(u_i)$ denote
the number of those vertices in $\mathcal{V}_0$ which are
at distance $r$ from the $i$-th saturated vertex $u_i$. By this definition, and our assumption on $|\mathcal{V}_0| = o(N^2)$ we have
\begin{equation}\label{As1}
 \sum_{r=1}^N n_{i,r}^0 = |\mathcal{V}_0|= M_1K= o(N^2)
\end{equation}
for any $i$. 
Hence, the number of vertices activated at step $i$ 
at distance $r$ from the $i$-th explored vertex, which we denote 
$\tilde{X}_{i,r} $,
 has the following distribution
\begin{equation*}
	\tilde{X}_{i,r} \in \Bin (N_r-n_{i,r}^0-\tilde{U}_{i-1,r}(u_i),p_r),
\end{equation*}
and the total number of vertices activated at the $i$-th step is given by
\begin{equation*}
	\tilde{X}_i = \sum_{r=1}^N \tilde{X}_{i,r} .
\end{equation*}

 Using these definitions we derive
\begin{equation}
\P \{ |\tilde{C}_{u}|>k \} \geq \P \{ \tilde{X}_1+\tilde{X}_2+\dots +\tilde{X}_t>t-1, \, \forall t\leq k-1 \}.
\label{eq:bp_ineq}
\end{equation}
To approximate the distribution of $\tilde{X}_i$ is the last formula, let us define, given numbers 
$0 \leq k_{i,r} \leq N_r-n_{i,r}^0$, $r=1,\dots,N$, $i=1,\dots, k-1$, such that 
\begin{equation}\label{As2}
\sum_{r=1}^N k_{i,r} \leq k, 
\end{equation}
the following  Poisson random variables 
\begin{equation*}
\tilde{Z}_{i,r} \in \Po((N_r-n_{i,r}^0-k_{i,r})p_r).
\end{equation*}
Then
\begin{equation}\label{O23}
\tilde{Z}_i = \sum_{r=1}^N \tilde{Z}_{i,r} \in \Po(\lambda_{i,N}),
\end{equation}
where
\begin{equation}\label{lain}
\lambda_{i,N}=\sum_{r=1}^N (N_r-n_{i,r}^0-k_{i,r})p_r.
\end{equation}
 
To simplify the notations define the event 
\begin{equation}\label{DefF}
F_N=\{\tilde{U}_{i,r}=k_{i,r}, \quad \sum_{r=1}^N k_{i,r}\leq k,\quad  \forall i\leq k-1, \forall r=1,\dots ,N \}
\end{equation}
and consider 
\begin{equation*}
\begin{split}
\P(\tilde{Z}_i>k| F_N)&=\P(\tilde{Z}_i>k, \tilde{Z}_i=\tilde{X}_i| F_N)\\
&+\P(\tilde{Z}_i>k, \tilde{Z}_i\neq \tilde{X}_i| F_N)\\
&\leq \P(\tilde{X}_i>k| F_N)+ \P( \tilde{Z}_i\neq \tilde{X}_i| F_N).
\end{split}
\end{equation*}
Note that
\begin{equation*}
\begin{split}
 \P\{\tilde{X}_i \neq \tilde{Z}_i |F_N\} & = \P\{\sum_{r=1}^{N} \tilde{X}_{i,r} \neq \sum_{r=1}^{N}\tilde{Z}_{i,r}|F_N \}\\
&\leq \P\{\bigcup_{r=1}^N \{\tilde{X}_{i,r} \neq \tilde{Z}_{i,r} \}|F_N \}\\
&\leq \sum_{r=1}^{N} \P\{\tilde{X}_{i,r} \neq \tilde{Z}_{i,r}|F_N \}.
\end{split}
\end{equation*}
We shall use the following result from \cite{rvdh2}.

\begin{lemma}\label{pobin}\cite{rvdh2} (Poisson limit for binomial random variable)
$X\eqd \Bin (n,\lambda/n)$, $Y\eqd \Po (\lambda)$ then the following holds
\[
\P(X\neq Y)\leq \frac{\lambda^2}{n}.
\]
\end{lemma}

By Lemma \ref{pobin} we have
\begin{equation*}
	\P\{\tilde{X}_{i,r} \neq \tilde{Z}_{i,r} |F_N\} \leq p_r^2(N_r-n_{i,r}^0-k_{i,r}),
\end{equation*}
which gives us
\begin{equation*}
\begin{split}
\P\{\tilde{X}_i \neq \tilde{Z}_i |F_N\} & \leq \sum_{r=1}^{N}p_r^2(N_r-n_{i,r}^0-k_{i,r})\\
&=\frac{c^2}{N^2}\sum_{r=1}^{N}\frac{1}{r^2}(N_r-n_{i,r}^0-k_{i,r})
= O\left(\frac{\log N}{N^2}\right),
\end{split}
\label{eq:neq}
\end{equation*}
uniformly in $i$.

Next we consider
\begin{equation*}
	\begin{split}
		\P\{ \tilde{Z}_1 + \dots + \tilde{Z}_{t} & > t-1, \forall t \leq k-1 \} 
\\ 
		& \leq \P\{ \tilde{X}_1 + \dots + \tilde{X}_{t} > t-1, \forall t \leq k-1|F_N \}\\ 
		& + \sum_{s=1}^{k-1} \P\{\tilde{X}_s \neq \tilde{Z}_s |F_N\}.
	\end{split}
\label{eq:O21}
\end{equation*}
Note that by  (\ref{eq:neq})
\begin{equation}
	\varepsilon_k(N): = \sum_{s=1}^{k-1} \P\{\tilde{X}_s \neq \tilde{Z}_s |F_N\} = O\left(\left(\frac{\log N}{N}\right)^2\right),
\label{eq:error_up}
\end{equation}
therefore (\ref{eq:O21}) yields
\begin{equation}
	\P\{ \tilde{X}_1 + \dots + \tilde{X}_{t} > t-1, \forall t \leq k-1 |F_N\} 
\label{eq:O22}
\end{equation}
\[
\geq \P\{ \tilde{Z}_1 + \dots + \tilde{Z}_{t}  > t-1, \forall t \leq k-1 \} - \varepsilon_k(N).
\]

In the following we construct a random variable $\tilde{Z}^-$ 
is stochastically smaller than any of 
the random variables $\tilde{Z}_i$, $k\leq i \leq k.$

First, using the assumption (\ref{As1}) and  Lemma \ref{lem:subtract} we derive
\begin{equation}\label{1lin}
\sum_rn_{i,r}^0p_r=\frac{c}{N}\sum_r\frac{n_{i,r}^0}{r} =o(1).
\end{equation}
From now on we shall assume that 
\begin{equation}\label{As3}
k = a \log N
\end{equation}
for some positive $a$. Under this assumption 
we have 
\begin{equation}\label{2lin}
\sum_rk_{i,r}p_r=\frac{c}{N}\sum_{r=1}^{N}\frac{k_{i,r}}{r}\leq \frac{c}{N}\sum_rk_{i,r}=\frac{kc}{N}=\frac{ac\log N}{N} =o(1).
\end{equation}
Hence, from (\ref{1lin}) and  (\ref{2lin}) we obtain the following bound for 
$\lambda_{i,N}$ defined in (\ref{lain}) 
\begin{equation}\label{lln}
\lambda_{i,N}=\sum_{r=1}^N (N_r-n_{i,r}^0-k_{i,r})p_r\geq 
\sum_{r=1}^N N_r p_r +o_i(1),
\end{equation}
where $o_i(1)$ might depend on $i$.
Note that by (\ref{ln}) 
\begin{equation}\label{O27}
\sum_{r=1}^N N_r p_r =\lambda +o(1).
\end{equation}
Hence, for any (constant) 
\begin{equation}\label{O28}
\lambda '<\lambda
\end{equation}
(\ref{lln}) together with (\ref{O27})  yields the following uniform in $i\leq k$ bound
\begin{equation}\label{O29}
\lambda_{i,N}>\lambda '.
\end{equation}

Recall that $\tilde{Z}_i\in \Po(\lambda_{i,N})$ by (\ref{O23}). Therefore (\ref{O29})
allows us to construct independent 
 $\tilde{Z}_i^-\in \Po(\lambda ')$, $1\leq i \leq k$, such that
\begin{equation}\label{O24}
\tilde{Z}_i^- \leq \tilde{Z}_i
\end{equation}
$a.s.$ for each $i$.
Now we can  derive  the following bound
\begin{equation}
\begin{split}
	\P(\tilde{Z}_1  +\dots, \tilde{Z}_{t} & > t-1, t=1, \dots, k-1) \\
	& \geq \P(\tilde{Z}_1^- + \dots + \tilde{Z}_{t}^- > t-1, t=1, \dots, k-1)\\
	& = \P\{{\cal T}\geq k \},
\end{split}
\label{O25}
\end{equation}
where  ${\cal T}$ denotes the total progeny of a 
branching process with offspring distribution Po$(\lambda')$.
Substituting this bound into (\ref{eq:O22}) we obtain
\begin{equation}
	\P\{ \tilde{X}_1 + \dots + \tilde{X}_{t} > t-1, \forall t \leq k-1 |F_N \} \geq 
\P \{ {\cal T} \geq k \} - \varepsilon_k(N),
\label{O26}
\end{equation}
where the right-hand side is uniform in  $F_N $ (where we assume only 
 conditions in eqs. (\ref{As2}) and  (\ref{As3})). This yields
\begin{equation*}
	\P\{ \tilde{X}_1 + \dots + \tilde{X}_{t} > t-1, \forall t \leq k-1 \} \geq \P\{ {\cal T} \geq k \} - \varepsilon_k(N),
\end{equation*}
and therefore, by  eq. (\ref{eq:bp_ineq}),
\begin{equation}
	\P\{|\tilde{C}_u| \leq  k \}\leq 1-\P\{ {\cal T}  \geq k\}+\varepsilon_k(N).
\label{eq:C_branch_bound}
\end{equation}

Using a well known formula for the distribution of the progeny of a branching process (see e.g. \cite{or}), we compute
\begin{equation*}
	\P\{{\cal T} \geq k  \}=\sum_{j=k}^{\infty}\frac{(\lambda' j)^{j-1}}{j!}e ^{\lambda' j}\geq \frac{(\lambda' k)^{k-1}}{k!}e ^{\lambda' k},
\end{equation*}
which, together with the Stirling formula, gives us
\begin{equation}\label{T}
	\P\{{\cal T}  \geq k  \} \geq \frac{1}{\sqrt{2 \pi}\lambda'}\frac{1}{ k^{3/2}} e^{-\alpha k}\left( 1+O\left(\frac{1}{k}\right) \right),
\end{equation}
where 
\begin{equation}\label{alf}
	\alpha = \lambda'-1-\log \lambda'.
\end{equation}

Substituting  eq. (\ref{T}) into  eq. (\ref{eq:C_branch_bound}), we obtain for
$k=a \log N$ that 
\begin{equation}\label{O30}
\begin{split}
	\P \{ |\tilde{C}_u| \leq a \log N \} & \leq 
	1	- \frac{1}{A_N}  \left(1+o\left(1\right)\right)+O \left( \left( \frac{\log N}{N}\right)^2 \right),\\
\end{split}
\end{equation}
where
\begin{equation}
	A_N = \sqrt{2\pi }\lambda' (a \log N)^{3/2} N^{a \alpha}.
\end{equation}
Choose now arbitrarily a constant 
\begin{equation}\label{a1}
a<\frac{2}{\alpha}.
\end{equation}
Then, eq. (\ref{O30}) yields
\begin{equation}\label{O31}
	\P \{ |\tilde{C}_u| \leq a \log N \}  \leq  1- \frac{1}{A_N}  \left(1+o\left(1\right)\right).
\end{equation}
Observe that the value on the right in the last formula is uniform in the choice of the set ${\mathcal V}_0$ and vertex $u$. Therefore we can use the bound in eq. (\ref{O31}) 
to obtain, from eq. (\ref{c1l}),
\begin{equation}\label{O32}
	\P\{ C \leq a \log N \} \leq \left(1- \frac{1}{A_N} \left(1+o\left(1\right)\right)\right)^{M_1}.
\end{equation}
Finally,  we choose
\[M_1=A_N\log N \gg A_N,\]
which by (\ref{a1}) satisfies as well  the condition  (\ref{As1}), i.e., 
\[M_1 K = o(N^2),\]
where 
\[K=\left(\frac{2}{\lambda-1-\log \lambda} +\varepsilon\right)\log N .\]
With this choice of $M_1$ bound  (\ref{O32}) implies
\begin{equation}\label{O33}
\P \{ C \leq a \log N \}  = o(1)
\end{equation}
for any fixed constant (see eqs(\ref{a1}) and (\ref{alf}))
\[
a<\frac{2}{\alpha}=\frac{2}{\lambda'-1-\log \lambda'}.
\]
Since by (\ref{O28}) here we can choose any $\lambda'<\lambda$, it follows that 
 (\ref{O33}) holds for any 
\[
a<\frac{2}{\lambda-1-\log \lambda}.
\]
This proves (\ref{LLC upp}), and therefore part $i)$ of Theorem \ref{T1} is proved. 
\end{proof}

This completes the proof of Theorem \ref{T1}, since part $ii)$ follows by Theorem \ref{T2}.

\subsection{The idea of proof  of  Theorem \ref{T3}}

A proof of Theorem \ref{T3} can
 run precisely the lines of the proof of part $i)$ of Theorem \ref{T1} in a 
combination with the proof of a corresponding result for the rank-1 model (\ref{pr1}) given in \cite{T2}. Therefore, we can omit the details here.


\begin{thebibliography}{10}

\bibitem{AN}
M.~Aizenman and C.~M. Newman.
\newblock Discontinuity of the percolation density in one dimensional $1/|x
  -y|^2$ percolation models.
\newblock {\em Comm. Math. Phys.}, 107:611--647, 1986.


\bibitem{A} D.~Aldous.
\newblock Brownian excursions, critical random graphs and the multiplicative coalescent. \newblock {\em  Ann. Probab.}, 25: 812--854, 1997.



\bibitem{avin}
C.~Avin.
\newblock Distance graphs: from random geometric graphs to {Bernoulli} graphs
  and between.
\newblock In {\em DIALM-POMC}, 2008.


\bibitem{BHL}  S.Bhamidi,   R. van der Hofstad,  and  J. van Leeuwaarden.
\newblock Scaling limits for critical inhomogeneous random graphs with finite
third moments.
\newblock {\em Electron. J. Probab.}, 15: 1682--1702, 2010. 

\bibitem{BHL1} S. Bhamidi, R. van der Hofstad  and  J. van Leeuwaarden.
\newblock Novel scaling limits for critical inhomogeneous random graphs.
\newblock {\em   Ann. Probab.}, 40: 2299-2361, 2012.




\bibitem{irg}
B.~Bollob\'as, S.~Janson, and O.~Riordan.
\newblock The phase transition in inhomogeneous random graphs.
\newblock {\em Random Struct. Algor.}, 31:3--122, 2007.

\bibitem{sp}
B.~Bollob\'as, S.~Janson, and O.~Riordan.
\newblock Spread-out percolation in $\mathbb{R}^d$.
\newblock {\em Random Struct. Algor.}, 31:239--246, 2007.

\bibitem{bp}
M.~Bradonji\'{c} and I.~Saniee.
\newblock Bootstrap percolation on random geometric graphs.
\newblock {\em Probab. Eng. Inform. Sc.}, 28:169--181, 2014.

\bibitem{rvdh}
M.~Deijfen, R.~van~der Hofstad, and G.~Hooghiemstra.
\newblock Scale-free percolation.
\newblock {\em Ann. I. H. Poincare-PR}, 49:817--838, 2013.

\bibitem{J2}
S.~Janson.
\newblock The largest component in a subcritical random graph with a power law
  degree distribution.
\newblock {\em Ann. Appl. Probab.}, 18:1651--1668, 2008.

\bibitem{koz}
S.~Janson, R.~Kozma, M.~Ruszink\'o, and Y.~Sokolov.
\newblock Bootstrap percolation on a random graph coupled with a lattice.
\newblock arXiv:1507.07997v2, 2015.

\bibitem{lindvall}
T.~Lindvall.
\newblock {\em Lectures on the coupling method}.
\newblock Wiley-Interscience Publication, 1992.

\bibitem{or}
R.~Otter.
\newblock The multiplicative process.
\newblock {\em Ann. Math. Statist.}, 20:206--224, 1949.

\bibitem{penrose}
M.~Penrose.
\newblock On the spread-out limit for bond and continuum percolation.
\newblock {\em Ann. Appl. Probab.}, 3:253--276, 1993.

\bibitem{P}
M.~Penrose.
\newblock {\em Random geometric graphs}.
\newblock Oxford Studies in Probability, 5. Oxford University Press, 2003.

\bibitem{T2}
T.~S. Turova.
\newblock The size of the largest component below phase transition in
  inhomogeneous random graphs.
\newblock {\em Comb. Probab. Comp.}, 20:131--154, 2010.

\bibitem{T}
T.~S. Turova.
\newblock The emergence of connectivity in neuronal networks: from bootstrap
  percolation to auto-associative memory.
\newblock {\em Brain Res.}, 1434:277-284, 2012.

\bibitem{T1}
T.~S. Turova.
\newblock Asymptotics for the size of the largest component scaled to ``log n''
  in inhomogeneous random graphs.
\newblock {\em Ark. Mat.}, 51:371--403, 2013.

\bibitem{T11}   T.~S. Turova.
\newblock Diffusion approximation for the components
    in critical inhomogeneous random graphs of rank 1. 
\newblock {\em Random Struct. Alg.}, 43:486 -- 539, 2013. 


\bibitem{rvdh2}
R.~van~der Hofstad.
\newblock Random graphs and complex networks, {Vol. I}.
\newblock http://www.win.tue.nl/~rhofstad/NotesRGCN.pdf, 2016.

\end{thebibliography}
\end{document}